\documentclass[11pt]{article}
\pdfoutput=1
\usepackage{amsmath, amsfonts, amsthm, amssymb, hyperref, setspace, geometry, mathrsfs}
\title{\textbf{The influence of weakly $S\Phi$-supplemented subgroups on fusion systems of finite groups}\thanks{\footnotesize \scriptsize\emph{E-mail address:}
      zsmcau@cau.edu.cn\,(S. Zhang) }}
 
\author{Shengmin Zhang \\
\quad
\\
{\small College of Science,
China Agricultural University,
Beijing 100083, China}}
\date{}
\newtheorem{theorem}{Theorem}[section]

\newtheorem{lemma}[theorem]{Lemma}
\newtheorem{corollary}[theorem]{Corollary}
\theoremstyle{definition}
\newtheorem{definition}[theorem]{Definition}

\allowdisplaybreaks

\expandafter\let\expandafter\oldproof\csname\string\proof\endcsname
\let\oldendproof\endproof
\renewenvironment{proof}[1][\proofname]{%
  \oldproof[\bfseries\scshape #1]%
}{\oldendproof}

\usepackage{stackengine,scalerel}
\stackMath
\def\trianglelefteqslant{\ThisStyle{\mathrel{%
  \stackinset{r}{.75pt+.15\LMpt}{t}{.1\LMpt}{\rule{.3pt}{1.1\LMex+.2ex}}{\SavedStyle\leqslant}%
}}}
\renewcommand{\unlhd}{\trianglelefteqslant}

\renewcommand{\leq}{\leqslant}
\renewcommand{\geq}{\geqslant}
\begin{document}
\maketitle
\begin{abstract}
Let $G$ be a finite group and $H$ be a subgroup of $G$. Then $H$ is called a weakly $S\Phi$-supplemented subgroup of $G$, if there exists a subgroup $T$ of $G$ such that $G =HT$ and $H \cap T \leq \Phi (H) H_{sG}$, where $H_{sG}$ denotes the subgroup of $H$ generated by all subgroups of $H$ which are $S$-permutable in $G$. Let $p$ be a prime, $S$ be a $p$-group and $\mathcal{F}$ be a saturated fusion system over $S$. Then $\mathcal{F}$ is said to be supersolvable, if there exists a series of $S$, namely $1 = S_0 \leq S_1 \leq \cdots \leq S_n = S$, such that $S_{i+1}/S_i$ is cyclic, $i=0,1,\cdots, n-1$, $S_i$ is strongly $\mathcal{F}$-closed, $i=0,1,\cdots,n$. In this paper, we investigate the structure of fusion system $\mathcal{F}_S (G)$ under the assumption that certain subgroups of $S$ are weakly $S\Phi$-supplemented in $G$, and obtain several new characterizations of   supersolvability of $\mathcal{F}_S (G)$.
\end{abstract}
\section{Introduction}
All groups considered in this paper will be finite. Let $G$ be a finite group and $H$ be a subgroup of $G$. Recall that $H$ is said to be  complemented in $G$, if there exists a subgroup $K$ such that $G =HK$, and $H \cap K =1$. In {{\cite{WY}}}, Wang introduced the following concept, which is regarded as one of the originations of generalised complementarity: $H$ is said to be $c$-supplemented ($c$-normal) in $G$, if there exists a subgroup $K$ of $G$ such that $G = HK$ and $H \cap K \leq H_G$. Clearly, if $H$ is complemented in $G$, then $H$ is $c$-supplemented in $G$. In {{\cite{GL}}}, Guo and Lu introduced the following concept: $H$ is called a $SS$-supplemented subgroup of $G$, if there exists a subgroup $K$ of $G$ such that $G = HK$ and $H \cap K$ is $S$-permutable in $K$, where the definition of $S$-permutable is as follows:
\begin{definition}
Let $G$ be a finite group and $A$ be a subgroup of $G$. We say $A$ is $S$-permutable ($S$-quasinormal) in $G$, if $AQ=QA$ for all Sylow subgroups $Q$ of $G$.
\end{definition}
With the definition above, for a fixed subgroup $H$ of $G$, we denote $H_{sG}$ the subgroup of $H$ generated by all subgroups of $H$ which are $S$-permutable in $G$. In {{\cite{AS}}}, Skiba introduced the following concept: A subgroup $H$ of $G$ is said to be weakly $s$-supplemented in $G$, if there exists a subgroup $T \leq G$ such that $G = HT$ and $H \cap T \leq H_{sG}$. Then a way of generalization comes into our mind, that is to change the restriction of $H \cap T$. For example, in {{\cite{WU}}}, Z. Wu {\it et al.} introduced the following definition: 
\begin{definition}
Let $G$ be a finite group and $A$ be a subgroup of $G$. Then $A$ is said to be a $S\Phi$-supplemented subgroup of $G$, if there exists a subgroup $T$ of $G$ such that $G =AT$ and $A \cap T \leq \Phi (A)$.
\end{definition}
Now, we want to generalise the concept of $S\Phi$-supplementarity. Combining with one way of generalization introduced above, we may change the restriction of $H \cap T$ from $\Phi (H)$ to $\Phi (H) H_{sG}$. That is exactly the concept introduced by {{\cite{MA}}}:
\begin{definition}
Let $G$ be a finite group and $A$ be a subgroup of $G$. Then $A$ is called a weakly $S\Phi$-supplemented subgroup of $G$, if there exists a subgroup $T$ of $G$ such that $G =AT$ and $A \cap T \leq \Phi (A) A_{sG}$. 
\end{definition}
Today, there are a lot of generalised supplementarities, which are widely researched by many authors. A natural question comes into our mind, that is to compare these supplementarities. For example, if $H$ is complemented in $G$, then $H$ satisfies almost all supplementarities like $SS$-supplemented, weakly $S \Phi$-supplemented and so on. If we denote this relation by $>$, i.e. complemented $>$ weakly $S \Phi$-supplemented, then we have the following conclusion:
\begin{align*}
&complemented > weakly ~\Phi\text{-}supplemented > weakly ~c\Phi\text{-}supplemented>weakly~ S\Phi\text{-}supplemented.\\
&complemented > c\text{-}supplemented>weakly ~c\Phi\text{-}supplemented>weakly~ S\Phi\text{-}supplemented.
\end{align*}
Where weakly $c\Phi$-supplemented property can be defined by changing the restriction of $H \cap T$ into $\Phi (H)H_G$. In {{\cite{WU}}}, Z. Wu {\it et al.} obtained the characterizations of $p$-nilpotency under the assumption that certain subgroups are weakly $\Phi$-supplemented:
\begin{theorem}[{{\cite[Theorem 3.1]{WU}}}]\label{1.4}
Let $N \unlhd G$ such that $G/N$ is $p$-nilpotent, where $p$ is the smallest prime divisor of $|G|$. Suppose that every cyclic subgroup of $N$ of order $4$ is weakly $S\Phi$-supplemented in $G$ and every
minimal subgroup of $N$ of order $p$ lies in $Z_{\mathfrak{N}_p} (G)$. Then $G$ is $p$-nilpotent.
\end{theorem}
\begin{theorem}[{{\cite[Theorem 3.4]{WU}}}]\label{1.5}
Let $N$ be a normal subgroup of $G$ such that $G/N$ is $p$-nilpotent, where $p$ is the smallest prime divisor of $|G|$. Suppose that every cyclic subgroup of $N$ with order $p$ or $4$ is weakly $S\Phi$-supplemented in $G$. Then $G$ is $p$-nilpotent.
\end{theorem}
In the first part of this paper, we investigate the influence of weakly $S\Phi$-supplemented subgroups on the structure of finite groups, and obtain the following results compared to Theorem \ref{1.4} and \ref{1.5}. 
\begin{theorem}\label{3.1}
Let $G$ be a group with $N \unlhd G$ such that $G/N$ is $p$-nilpotent. Suppose that every minimal subgroup of $N$ of order $p$ is contained in $Z(G)$, and every cyclic subgroup of $N$ of order $4$(if $p=2$) is weakly $S\Phi$-supplemented in $G$. Then $G$ is $p$-nilpotent.
\end{theorem}
\begin{theorem}\label{3.2}
Let $G$ be a finite group with a normal subgroup $N$ such that $G/N$ is nilpotent. Suppose that every minimal subgroup of $F^{*} (N)$ is contained in $Z(G)$ and that every cyclic subgroup of order $4$ is weakly $S\Phi$-supplemented in $G$. Then $G$ is nilpotent.
\end{theorem}
In the first part of this paper, we also obtain another characterization of $p$-nilpotency under the assumption that certain subgroups are weakly $\Phi$-supplemented in Theorem \ref{3.3}, and show the influence of weakly $\Phi$-supplementarity on the structure of the chief factors of $G$ in Theorem \ref{3.4}. 

Now, we are going to give a brief introduction to the basic theory of fusion systems, and begin to introduce our results on the characterization of the structure of $\mathcal{F}_S (G)$ under the assumption that certain subgroups of $S \in {\rm Syl}_p (G)$ are weakly $S \Phi$-supplemented in $G$. Let $S$ be a Sylow $p$-subgroup of $G$, where $p$ is a prime divisor of $|G|$. Then the fusion system of $G$ over $S$, named $\mathcal{F}_S (G)$, is a fusion category over $S$ which is defined as follows:
\begin{itemize}
\item[(1)] The object of $\mathcal{F}_S (G)$ is the set of all subgroups of $S$.
\item[(2)] For any $P,Q \leq S$, ${\rm Mor}_{\mathcal{F}_S (G)} (P,Q)=\lbrace \phi\,|\,\phi:P\rightarrow Q,\,p \mapsto p^g,\,P^g \leq Q,\,g \in G \rbrace  $.
\end{itemize}
One can easily find that $\mathcal{F}_S (G)$ is exactly a saturated fusion system over $S$ by {{\cite[Theorem 2.3]{AK}}}. As is known to all, the structure of $\mathcal{F}_S (G)$ has a strong relationship with the structure of $G$. Hence some structures of finite groups can be generalized into the fusion system $\mathcal{F}_S (G)$. Recall that $G$ is said to be supersolvable, if there exists a chief series, namely
$$1 = N_0 \leq N_1 \leq \cdots \leq N_t = G,$$
such that $N_{i+1}/N_i$ is cyclic, $i=0,1,\cdots,t-1$. As a natural way of generalization, we want to define a similar structure in $\mathcal{F}$, where $\mathcal{F}$ is a fusion system over a $p$-group $S$. Since the objects of $\mathcal{F}$ are exactly the subgroups of $S$, we may restrict the chief series of $G$ into a series of $S$. Note that the normality of subgroups $H$ of $S$ in $G$ represents the invariance of $H$ under the morphisms induced by conjugation of $G$, we may change the normality of $H$ in $G$ into invariance of $H$ under the morphisms in $\mathcal{F}$. Then one can easily find that the invariance of $H$ under the morphisms in $\mathcal{F}$ exactly suits the concept of weakly $\mathcal{F}$-closed property, hence we give the following definition which was introduced by N. Su in {{\cite{SN}}}.
\begin{definition}
Let $\mathcal{F}$ be a saturated fusion system over a $p$-group $S$. Then $\mathcal{F}$ is called supersolvable, if there exists a series of subgroups of $S$, namely:
$$1 = S_0 \leq S_1 \leq \cdots \leq S_n =S,$$
such that $S_i$ is strongly $\mathcal{F}$-closed, $i=0,1,\cdots, n$, and $S_{i+1}/S_i$ is cyclic for any $i=0,1,\cdots,n-1$.
\end{definition}

Now we would like to introduce some concepts which are useful for us to discover the structure of $\mathcal{F}_S (G)$. Let $S$ be a $p$-group and $P$ be a subgroup of $S$. Suppose that $\mathcal{F}$ is a fusion system over $S$. Then $P$ is called $\mathcal{F}$-centric, if $C_S (Q) = Z(Q)$ for all $Q \in P^{\mathcal{F}}$, where $P^{\mathcal{F}}$ denotes the set of all subgroups of $S$ which are $\mathcal{F}$-conjugate to $P$. $P$ is said to be fully normalized in $\mathcal{F}$, if $|N_S (P)| \geq |N_S(Q)|$ for all $Q \in P^{\mathcal{F}}$. $P$ is said to be $\mathcal{F}$-essential, if $P$ is $\mathcal{F}$-centric and fully normalized in $\mathcal{F}$, and ${\rm Out}_{\mathcal{F}} (P)$ contains a strongly $p$-embedded subgroup (see {{\cite[Definition A.6]{AK}}}). Now we are ready to introduce the following concept.
\begin{definition}
Let $p$ be a prime, $\mathcal{F}$ be a saturated fusion system on a finite group $S$. Let 
$$\mathcal{E}_{\mathcal{F}} ^{*}:=\lbrace Q \leq S\,|\,Q ~\text{is $\mathcal{F}$-essential, or $Q =S$}\rbrace. $$
\end{definition}
Let $G$ be  a finite group and $S$ be a Sylow $p$-subgroup of $G$. In many cases, if certain subgroups of $S$ satisfy certain properties, then $G$ is $p$-nilpotent or $p$-supersolvable, just like Theorem \ref{3.1} and \ref{3.2}. On the other hand, on the assumption that a smaller collection of  subgroups of $S$ satisfy some properties, for example, they are weakly $S\Phi$-supplemented in $G$, then $G$ is not necessarily a $p$-supersolvable subgroup. But actually, this does not mean that there is nothing to do with the structure of the finite group $G$. Notice that the fusion system $\mathcal{F}_S (G)$ is just a description of the structure of $G$, we wonder if we can characterize the structure of $\mathcal{F}_S (G)$. Fortunately, we obtain the following theorem to characterize the structure of $\mathcal{F}_S (G)$ under the assumption that certain subgroups of $S$ are weakly $S\Phi$-supplemented in $G$, which highlights the paper considerably.
\begin{theorem}\label{66666}
Let $G$ be a finite group and $S$ a Sylow $p$-subgroup of $G$, where $p$ is the smallest prime divisor of $|G|$. Suppose that ${\rm exp} (S) = p$, and every minimal subgroup of $S$  is weakly $S\Phi$-supplemented in $G$, then $\mathcal{F}_S (G)$ is supersolvable.
\end{theorem}

 \section{Preliminaries}
We will show in this section that the inheritance property of weakly $S\Phi$-supplemented subgroups is similar to those of other supplemented subgroups like $S\Phi$-supplemented subgroups in {{\cite[Lemma 2.1]{LZ}}} and weakly s-supplemented subgroups in {{\cite[Lemma 2.5]{MA2}}}.
\begin{lemma}[{{\cite[Lemma 2.5]{MA}}}]\label{1}
Let $G$ be a finite group, $H \leq K \leq G$, and $N \unlhd G$. Then the following hold:
\begin{itemize}
\item[(1)] If $H$ is a weakly $S\Phi$-supplemented subgroup of $G$, then $H$ is a weakly $S\Phi$-supplemented subgroup of $K$.
\item[(2)] If $N \leq H$ and $H$ is a weakly $S\Phi$-supplemented subgroup of $G$, then $H/N$ is a weakly $S\Phi$-supplemented subgroup of $G/N$.
\item[(3)] If $(|N|,|H|)=1$, and $H$ is a weakly $S\Phi$-supplemented subgroup of $G$, then $NH/N$ is a weakly $S\Phi$-supplemented subgroup of $G/N$.
\end{itemize}
\end{lemma}
We say a group $H$ is quasisimple if $H' =H$, and either $N \leq Z(H)$ or $H =N$ for any $N \unlhd H$. Let $G$ be a finite group and $H \leq G$. It is well known that a quasisimple group $H$ is a component of $G$, if $H$ is subnormal in $G$. Then we denote the subgroup of $G$ generalised by all components of $G$ by $E(G)$. It is easy to find that $E(G)~{\rm{char}}~G$, and $[E(G),F(G)]=1$. Then the generalised Fitting subgroup of $G$ is defined by $F^{*} (G):=E(G)F(G)$. Now we introduce some basic properties about $F^{*} (G)$, which will be widely used in our proofs.
\begin{lemma}[{{\cite[Chapter X]{HU3}}} and {{\cite[Lemma 4]{BA}}}]\label{2}
Let $G$ be a group.
\begin{itemize}
\item[(1)] If $N$ is a normal subgroup of $G$, then $F^{*} (N) = N \cap F^{*} (G)$.
\item[(2)] If $F^{*} (G)$ is soluble, then $F^{*} (G) = F(G)$.
\item[(3)] $F(G) \leq F^{*} (G) = F^{*} (F^{*} (G))$.
\item[(4)] Let $p$ be a prime and $P$ be a normal subgroup of $G$. Suppose that $P \leq Z(G)$, then $F^{*} (G/P) = F^{*} (G)/P$.
\end{itemize}
\end{lemma}
\begin{lemma}\label{111111}
Let $p$ be a prime and $\mathcal{F}$ be a saturated fusion system on a finite $p$-group $S$. Assume that the fusion system $N_{\mathcal{F}} (Q)$ is supersolvable for any $Q \in \mathcal{E}_{\mathcal{F}} ^{*}$, then $\mathcal{F}$ is supersolvable.\end{lemma}
\begin{proof}
Let $Q \in \mathcal{E}_{\mathcal{F}} ^{*}$. It follows from $N_{\mathcal{F}} (Q)$ is supersolvable and {{\cite[Proposition 1.3]{SN}}} that ${\rm Aut}_{N_{\mathcal{F}} (Q)} (Q) = {\rm Aut}_{\mathcal{F}} (Q)$ is $p$-closed. Hence, ${\rm Out}_{\mathcal{F}} (Q)$ is $p$-closed since ${\rm Out}_{\mathcal{F}} (Q)$ is a quotient group of ${\rm Aut}_{\mathcal{F}} (Q)$. By {{\cite[Proposition A.7 (c)]{AK}}}, we conclude that there is no subgroup $H$ of a $p$-closed finite group $G$ such that $H$ is strongly $p$-embedded with respect to $G$. Therefore ${\rm Out}_{\mathcal{F}} (Q)$ does not possess a strongly $p$-embedded subgroup, which implies that $Q$ is not $\mathcal{F}$-essential. Thus we get that $\mathcal{E}_{\mathcal{F}} ^{*} = \lbrace S \rbrace$. Now it indicates from {{\cite[Part I, Proposition 4.5]{AK}}} that $S$ is normal in $\mathcal{F}$. Hence the proof is complete since $N_{\mathcal{F}} (S) = \mathcal{F}$ is supersolvable by our hypothesis.
\end{proof}

\begin{lemma}[{{\cite[Lemma 2.9]{FJ}}}]\label{222222}
Let $G$ be a finite group, $p \in \pi(G)$, and $S$ be a Sylow $p$-subgroup of $G$. Suppose that for any proper subgroup $H$ of $G$ with $O_p (G) <S \cap H$ and $S \cap H \in {\rm Syl}_p (H)$, the fusion system $\mathcal{F}_{S \cap H} (H)$ is supersolvable. Assume additionally that $O_p (G) \leq Z_{\mathfrak{U}} (G)$. Then $\mathcal{F}_{S} (G)$ is supersolvable.
\end{lemma}
\section{Characterizations for $p$-supersolvability of finite groups}\label{3}
\begin{proof}[Proof of Theorem \ref{3.1}]
Assume that the theorem is false and let $G$ be a counterexample of minimal order. Now for the ease of reading we break the argument into separately stated steps.
\begin{itemize}
\item[\textbf{Step 1.}] $G$ is a minimal non-nilpotent group, $G = P  \rtimes Q$, where $P$ is the normal Sylow $p$-subgroup of $G$ with ${\rm{exp}}(P) = p$ or $4$ if $p=2$, $P / \Phi (P)$ is a chief factor of $G$. $Q$ is a Sylow $q$-subgroup of $G$.
\end{itemize}
Let $L$ be a proper subgroup of $G$. We conclude from Isomorphism Theorem that
$$L/L \cap N \cong LN/N \leq G/N,~\text{where}~G/N ~\text{is}~p\,\text{-nilpotent}.$$
By inheritance of $p$-nilpotency, it follows that $L/L \cap N$ is $p$-nilpotent. By our hypothesis and lemma \ref{1}(1), it yields that every cyclic subgroup of $L \cap N$ with order $4(p=2)$ is weakly $S\Phi$-supplemented in $L$. Since every minimal subgroup of $N$ of order $p$ is contained in $Z(G)$, and $Z(G) \cap L \leq Z(L)$, we have that every minimal subgroup of $N \cap L$ of order $p$ is contained in $Z(L)$. Hence it follows from $N \cap L \unlhd L$ that $L$ satisfies our hypothesis. Thus we have that $L$ is $p$-nilpotent and we conclude by the randomness of $L$ that $G$ is a minimal non-$p$-nilpotent group, i.e. a minimal non-nilpotent group. By {{\cite[Chapter IV, Theorem 5.4]{HU1}}} and {{\cite[Chapter IV, Theorem 3.4.11]{GW}}}, it indicates that  there exists a normal $p$-subgroup $P$ such that $G=P \rtimes Q$, where $Q$ is a Sylow $q$-subgroup of $G$ with $q \neq p$, $P / \Phi (P)$ is a chief factor of $G$, and ${\rm{exp}}(P)=p$ or $4(p=2)$.
\begin{itemize}
\item[\textbf{Step 2.}] There exists an element in $P$ of order $4$ and $p=2$.
\end{itemize}
It follows from {{\cite[Chapter IV, Theorem 3.4.11]{GW}}} that $P = G^{\mathfrak{N}}$, where $G^{\mathfrak{N}}$ denotes the smallest normal subgroup of $G$ such that $G /G^{\mathfrak{N}}$ is nilpotent. Hence we conclude that $P \leq N$. Assume that the statement is false, then ${\rm{exp}} (P) =p$ by step 1. By our hypothesis, $P \leq Z(G)$. Therefore we have that $[P,Q]=1$ and so $Q \unlhd G$, a contradiction to the fact that $G$ is not nilpotent and we are done.
\begin{itemize}
\item[\textbf{Step 3.}] Final contradiction.
\end{itemize}
Suppose firstly that $P$ is not cyclic. We predicate that there exists an element $x_0 \in P$ of order $4$ such that $\langle x_0 \rangle$ is not S-permutable in $G$. Assume that the statement is false, then every element $x$ of $P$ of order $4$ is S-permutable in $G$. Then for any $Q$ being Sylow subgroup of $G$, we have that $\langle x \rangle Q = Q \langle x \rangle$. Since $P$ is not cyclic, it follows that $Q \langle x \rangle < G$. By the fact that $G$ is minimal non-nilpotent, we have that $Q \langle x \rangle$ is nilpotent. Hence we conclude that $[\langle x \rangle, Q]=1$ holds for every $x \in P$ of order $4$. Since every element of $P$ of order $2$ is contained in $Z(G)$, it indicates from ${\rm{exp}}(P) = 4$ that $[P,Q]=1$, a contradiction to the fact that $G$ is not nilpotent. Therefore the statement is true and let $x_0$ be an element of $P$ of order $4$ such that $\langle x_0 \rangle$ is not S-permutable in $G$. Since all of the subgroups of $\langle x_0 \rangle$ are exactly $\langle x_0 \rangle$, $\langle x_0 ^2 \rangle$, 1, it yields that $\langle x_0 \rangle_{sG} \leq \langle x_0 ^2 \rangle$. By our hypothesis, $\langle x_0 \rangle$ is weakly $S\Phi$-supplemented in $G$. It follows immediately that there exists a subgroup $T \leq G$ such that $G = \langle x_0 \rangle T$, and $T \cap \langle x_0 \rangle \leq \Phi(\langle x_0 \rangle) \langle x_0 \rangle_{sG} \leq \langle x_0 ^2 \rangle$. As $P' <P$, $P' ~\text{char}~ P$, and $P/ \Phi (P)$ is a chief factor of $G$, we conclude that $P' \leq \Phi (P)$. Hence we have that $P/\Phi (P)$ is an abelian minimal normal subgroup of $G /\Phi (P)$. By $P/\Phi (P) \cdot T\Phi (P)/\Phi (P) = G/\Phi (P)$ and {{\cite[Chapter I, Theorem 1.7.1]{KS}}}, it indicates that $P/\Phi (P) \leq T\Phi (P)/\Phi (P)$ or $P/\Phi (P) \cap T\Phi (P)/\Phi (P)=1$. The former case suggests that $P=P \cap T \Phi (P) = \Phi (P) (P \cap T)$. Then $P \cap T = P$ and so $\langle x_0 \rangle \cap T = \langle x_0 \rangle$, a contradiction. Hence we conclude that $\Phi (P) = P \cap T \Phi (P)$, which implies that $\Phi (P) =\Phi (P) (P \cap T)$, i.e. $P \cap T \leq \Phi (P)$. Thus we have that $P = P \cap \langle x_0 \rangle T = \langle x_0 \rangle (P \cap T) = \langle x_0 \rangle \Phi (P)$. Therefore we get that $P = \langle x_0 \rangle$, a contradiction to our assumption that $P$ is not cyclic. Hence we have that $P$ is cyclic and $P$ is a $2$-group. By {{\cite[Chapter IV, Theorem 2.8]{HU1}}}, it yields that $G$ has a normal $2'$-subgroup. Hence we have that $G$ is nilpotent, a final contradiction and no such counterexample of $G$ exists.
\end{proof}

\begin{proof}[Proof of Theorem \ref{3.2}]
Suppose that the theorem is not true and let $G$ be a counterexample of minimal order. Let $M$ be a proper normal subgroup of $G$. We argue that $M$ satisfies our hypothesis. Again, we conclude from Isomorphism theorem that 
$$M/M \cap N \cong MN/N \leq G/N,~\text{where}~G/N ~\text{is nilpotent}.$$
By inheritance of nilpotency we get that $M/M \cap N$ is nilpotent. By lemma \ref{2}(1), we have that $F^{*} (M \cap N) \leq F^{*} (N)$. It follows from $Z(M) \leq Z(G)$ that every minimal subgroup of $F^{*} (M \cap N)$ is contained in $Z(M)$, every cyclic subgroup of $F^{*} (M \cap N)$ of order $4$ is weakly $S\Phi$-supplemented in $M$ by lemma \ref{1} (1). Now it yields that $M$ satisfies our hypothesis, hence by the choice of $G$ we conclude that $M$ is nilpotent. Therefore every proper normal subgroup of $G$ is nilpotent, and we have that $F(G)$ is the unique maximal normal subgroup of $G$. Now we predicate that $G =N = G^{\mathfrak{N}}$. Suppose that $N<G$, then both $N$ and $G/N$ are nilpotent and it indicates from lemma \ref{2}(2) that $F^{*} (N) = F(N) = N$. Now it is clear that $N$ satisfies the hypothesis of theorem \ref{3.1}, hence we have that $G$ is $p$-nilpotent for all primes $p$, i.e. $G$ is nilpotent, a contradiction. Therefore $N =G$. Now suppose again that $G^{\mathfrak{N}}<G$, it follows directly that both $G^{\mathfrak{N}}$ and $G / G^{\mathfrak{N}}$ are nilpotent. Hence we conclude from lemma \ref{2}(2) that 
$$F^{*} (G^{\mathfrak{N}}) = F(G^{\mathfrak{N}}) = G^{\mathfrak{N}} \leq F(G) \leq F^{*} (G) = F^{*} (N).$$
Therefore one can easily find that $G^{\mathfrak{N}}$ satisfies the hypothesis of theorem \ref{3.1}, i.e. $G$ is nilpotent, a contradiction as well. Therefore we get that $G = N = G^{\mathfrak{N}}$. Now let $p$ be the smallest prime dividing the order of $F^{*} (G)$, and $P$ be a Sylow $p$-subgroup of $F^{*} (G)$. Assume that $F^{*} (G) = G$, it follows from lemma \ref{2}(3) that $F^{*} (F^{*} (G)) = F^{*} (G) = F^{*} (N)$. Hence $F^{*} (G)$ satisfies the hypothesis of theorem \ref{3.1}, and we get that $G$ is nilpotent again, a obvious contradiction. Thus $F^{*} (G)$ is a proper normal subgroup of $G$ and we have that $P$ is normal in $G$ by nilpotency of $F^{*} (G)$. Now let $Q$ be an arbitrary Sylow $q$-subgroup of $G$ with $q \neq p$ be a prime. We predicate that $PQ$ is $p$-nilpotent. In fact, it is obvious that $PQ/P$ is $p$-nilpotent. Also, it follows from $P \leq F^{*} (G) = F^{*}(N)$ that every minimal subgroup of $P$ of order $p$ is contained in $Z(PQ)$. By lemma \ref{1}(1), we conclude that every cyclic subgroup of $P$ of order $4$ is weakly $S\Phi$-supplemented in $PQ$. Hence $PQ$ satisfies the hypothesis of theorem \ref{3.1} and so $PQ$ is $p$-nilpotent. Thus $Q \unlhd PQ$ and we have that $[P,Q]=1$. Hence we get that $Q \leq C_G (P)$, and by randomness of $Q$ it indicates that $O^p (G) \leq C_G (P)$. Note that $G/C_G (P)$ is a $p$-group, we have that $G^{\mathfrak{N}} \leq C_G(P)$, i.e. $C_G (P) = G$. It follows directly from lemma \ref{2}(4) that $P \leq Z(G)$, and $F^{*} (G/P) = F^{*} (G) /P$. By the minimality of $p$ we get that $2 \nmid |F^{*} (G) /P|$. Let $H/P$ be a minimal subgroup of $F^{*} (G)/P$. Since $P \in {\rm{Syl}}_p (F^{*} (G))$, we get that $H = RP$, where $R$ is a minimal subgroup of $F^{*} (G) = F^{*} (N)$, and so $R$ is contained in $Z(G)$. Thus $H/P = RP/P$ is contained in $Z(G/P)$. By the randomness of $H/P$, $G/P$ satisfies the hypothesis and so by the choice of $G$, it yields that $G/P$ is nilpotent. Hence we conclude that $G$ is nilpotent, a final  contradiction and no such counterexample of $G$ exists.
\end{proof}
\begin{theorem}\label{3.3}
Let $p$ be the smallest prime dividing the order of a group $G$ and let $P$ be a Sylow $p$-subgroup of $G$. Suppose that $p$ is odd, then $G$ is  $p$-nilpotent if and only if every cyclic subgroup of $P$ of order prime not having a supersoluble supplement in $G$ is weakly $S\Phi$-supplemented in $G$.
\end{theorem}
\begin{proof}Our proof is proceeded via the two parts.
\begin{itemize}
\item[\textbf{Step 1.}] Necessity of the proof.
\end{itemize}
For any subgroup $H$ of $G$, if $H$ is weakly $\Phi$-supplemented in $G$, then it follows from $\Phi (H) \leq \Phi (H) H_{sG}$ that $H$ is weakly $S\Phi$-supplemented in $G$. Applying {{\cite[Lemma 4.1]{LX}}}, we conclude that every cyclic subgroup of $P$ of prime order not having a supersoluble supplement is weakly $S\Phi$-supplemented in $G$, and the necessity of the proof has finished.
\begin{itemize}
\item[\textbf{Step 2.}] Sufficiency of the proof.
\end{itemize}
Suppose that the theorem is false and let $G$ be a counterexample, i.e. $G$ satisfies our hypothesis, but $G$ is not $p$-nilpotent. Then $G$ contains a minimal non-$p$-nilpotent subgroup $A$. Again, by {{\cite[Chapter IV, Theorem 5.4]{HU1}}} and {{\cite[Chapter IV, Theorem 3.4.11]{GW}}}, it follows that $A$ is minimal non-nilpotent possessing  four properties: (1) $A = A_p \rtimes A_q$, $A_p$ is a Sylow $p$-subgroup of $A$, $A_q $ is a cyclic Sylow $q$-subgroup of $A$, where $q \neq p$ is a prime. (2) $A_p = A^{\mathfrak{N}}$. (3) ${\rm{exp}}(A_p)=p$ or $4$. Since $p$ is odd, ${\rm{exp}}(A_p)=p$. (4) $A_p / \Phi (A_p)$ is a chief factor of $A$. By Sylow Theorem, we may assume that $A_p \leq P$. It indicates from lemma \ref{1}(1) that every cyclic subgroup of $A_p$ of  prime order not having a supersoluble supplement in $A$ is weakly $S\Phi$-supplemented in $A$. Let $x$ be a non-trivial element in $A_p$. Then $o(x)=p$. If $\langle x \rangle$ has a supersoluble supplement in $A$, then there exists a supersoluble subgroup $T \leq A$ such that $A = \langle  x\rangle T$. If $\langle x \rangle \cap T=\langle x \rangle$, it follows directly that $T =A$. Thus $A$ is supersoluble. Now let 
$$\Gamma:1 =P_0 <P_1 < \cdots <P_n = P < B_1 < \cdots <B_m =A$$
be a chief series of $A$. By generalised Jordan-Holder Theorem and the supersolubility of $A$, $P_i /P_{i-1}$, $i=1,2,\cdots,n$ is of prime order. Hence $P_{n-1}$ is a normal subgroup of $A$ and $N:=P_{n-1} A_q$ is a nilpotent proper subgroup of $A$. Note that $|A:N|=p$, it follows  that $N$ is normal in $A$ by the minimality of $p$. Thus we have that $A_q ~{\rm{char}}~ N \unlhd A$, i.e. $A$ is nilpotent, a contradiction. Hence we get that $\langle x \rangle \cap T=1$. It follows from $|\langle x \rangle|=p$ and {{\cite[Exercise 3.1.3]{KS}}} that $T \unlhd A$. Hence $A_q \unlhd A$, a contradiction. Thus $\langle x \rangle$ has no supersoluble supplement in $A$, which implies that $\langle x \rangle$ is weakly $S\Phi$-supplemented in $A$. Assume firstly that $A_p$ is not cyclic. Suppose that every cyclic subgroup of $A$ of order $p$ is S-permutable in $A$. Then $A_q \langle x \rangle = \langle x \rangle A_q < A$ for any cyclic subgroup $\langle x \rangle$ of $A$ of order $p$. Hence $A_q\langle x \rangle $ is nilpotent and so $[A_q,\langle x \rangle]=1$. Since ${\rm{exp}} (P)=p$, we conclude by the randomness of $\langle x \rangle$ that $[P,A_q]=1$, a contradiction. Thus there exists a cyclic subgroup $\langle x_0 \rangle$ of $A$ of order $p$ such that $\langle x_0 \rangle$ is not S-permutable in $A$. Since $\langle x_0 \rangle$ is weakly $S\Phi$-supplemented in $A$, there exists $T \leq A$ such that $A = \langle x_0 \rangle T$ and $T \cap \langle x_0 \rangle  \leq \Phi(\langle x_0 \rangle) \langle x_0 \rangle_{sA} =1$ as $\langle x_0 \rangle$ is not S-permutable in $A$. Now in view of the proof of theorem \ref{3.1}, we conclude that either $A_p \cap T \Phi(A_p) =A_p$ or $A_p \cap T \Phi(A_p) =\Phi(A_p)$. The former case suggests that $T \geq A_p \geq \langle x_0 \rangle$, which indicates that $T \cap \langle x_0 \rangle = \langle x_0 \rangle$, a contradiction. The later case suggests that $A_p \cap T \leq \Phi (A_p)$, which implies that $A_p = A_p \cap \langle x_0 \rangle T = \langle x_0 \rangle (A_p \cap T) = \langle x_0 \rangle \Phi (A_p)$. Thus we have that $\langle x_0 \rangle  = A_p$, a contradiction to our assumption that $A_p$ is not cyclic. Therefore $A_p$ is cyclic and we conclude from {{\cite[Chapter IV, Theorem 2.8]{HU1}}} and the minimality of $p$ that $A$ is nilpotent, a final contradiction and we are done.
\end{proof}
\begin{theorem}\label{3.4}
Let $P$ be a non-trivial normal $p$-subgroup of $G$, where $p$ is the smallest prime dividing the order of $G$. If ${\rm{exp}} (P) = p$, every minimal subgroup of $P$ not containing a supersoluble supplement in $G$ is weakly $S\Phi$-supplemented in $G$, then every chief factor of $G$ below $P$ is cyclic.
\end{theorem}
\begin{proof}
We predicate that $P/\Phi (P)$ is a normal subgroup of $G/\Phi (P)$ satisfying the hypothesis. Clearly we have ${\rm{exp}} (P/\Phi (P)) = p$. Let $H/\Phi (P)$ be a minimal subgroup of $P/\Phi (P)$. Then $H /\Phi (P) = \langle x \rangle \Phi (P)/\Phi (P)$, where $x \in H \setminus \Phi (P)$. It is obvious that $o(x) =p$. By our hypothesis, either $\langle x \rangle $ has a supersoluble supplement in $G$ or $\langle x \rangle$ is weakly $S\Phi$-supplemented in $G$. If $\langle x \rangle$ has a supersoluble supplement $T$ in $G$, we verify that $T\Phi (P)/\Phi (P)$ is a supersoluble supplement of $H/\Phi (P)$ in $G/ \Phi (P)$. Since $T <G$, it follows that $T \Phi (P) <G$. If $1 < T \Phi (P)/\Phi (P) \cap H/\Phi (P)$, then $H/\Phi (P) = T \Phi (P)/\Phi (P) \cap H/\Phi (P)$ by the choice of $H/\Phi (P)$. Hence $G/\Phi (P) = T \Phi (P)/\Phi (P) \cdot H/\Phi (P) = T \Phi (P)/\Phi (P)$, which implies that $T \Phi (P) =G$, a contradiction. Therefore $T\Phi (P)/\Phi (P)$ is a supersoluble supplement of $H/\Phi (P)$ in $G/ \Phi (P)$. If $\langle x \rangle$ is weakly $S\Phi$-supplemented in $G$, then there exists $T \leq G$ such that $G = \langle x \rangle T$, and $T \cap \langle x \rangle  \leq \Phi (\langle x \rangle) \langle x \rangle_{sG} = \langle x \rangle_{sG}$. We predicate that $H/\Phi (P)$ is weakly $S\Phi$-supplemented in $G/\Phi (P)$ and the verification is proceeded via the two parts.
\begin{itemize}
\item[\textbf{Step 1.}] The condition of $\langle x \rangle_{sG}=\langle x \rangle$.
\end{itemize}
It is easy to see that $G/\Phi (P) = H/\Phi (P) \cdot T/\Phi (P)$. Let $Q_0 \Phi (P)/\Phi (P)$ be an arbitrary Sylow $q$-subgroup of $G/\Phi (P)$, where $Q_0$ is a Sylow $q$-subgroup of $G$ with $q \neq p$. Then we conclude from $\langle x \rangle_{sG}=\langle x \rangle$ that $\langle x \rangle Q_0 = Q_0 \langle x \rangle$. Hence we have that 
\begin{align*}
Q_0\Phi (P) /\Phi (P) \cdot \langle x \rangle \Phi (P)/\Phi (P) & = Q_0 \langle x \rangle \cdot \Phi (P)/\Phi (P) =  \langle x \rangle Q_0 \cdot \Phi (P) /\Phi (P)\\
& =\langle x \rangle \Phi (P)/\Phi (P) \cdot Q_0\Phi (P) /\Phi (P). 
\end{align*} 
Now let $P_0 /\Phi (P)$ be a Sylow $p$-subgroup of $G/\Phi (P)$, where $P_0$ is a Sylow $p$-subgroup of $G$. Since $\langle x \rangle P_0 = P_0 \langle x \rangle$, it follows from the same method that 
$$P_0 /\Phi (P) \cdot \langle x \rangle \Phi (P)/\Phi (P) = \langle x \rangle \Phi (P)/\Phi (P) \cdot  P_0 /\Phi (P) .$$
By the randomness of $Q$ and $P_0$, we conclude that $\langle x \rangle \Phi (P)/\Phi (P) = \langle x \rangle \Phi (P)/\Phi (P)_{sG /\Phi(P)}$. Thus $H /\Phi (P) \cap T\Phi (P)/\Phi (P) \leq \Phi(H /\Phi (P)) \cdot  \langle x \rangle \Phi (P)/\Phi (P)_{sG /\Phi(P)} = H/\Phi (P)$, i.e. $H/\Phi (P)$ is weakly $S\Phi$-supplemented in $G/\Phi (P)$. 
\begin{itemize}
\item[\textbf{Step 2.}] The condition of $\langle x \rangle_{sG}=1$.
\end{itemize}
It follows from $T \cap \langle x \rangle \leq \Phi (\langle x \rangle) \langle x \rangle_{sG} =1$ that $T <G$. Hence we have that $ H/\Phi (P) \cap T/\Phi (P) = 1 \leq \Phi(H /\Phi (P)) \cdot  \langle x \rangle \Phi (P)/\Phi (P)_{sG /\Phi(P)}$, i.e. $H/\Phi (P)$ is weakly $S\Phi$-supplemented in $G/\Phi (P)$. 

Thus we conclude that $P /\Phi (P)$  is a normal subgroup of $G /\Phi (P)$ satisfying the hypothesis. By induction, we get that every chief factor of $G/\Phi (P)$ below $P /\Phi (P)$ is cyclic, i.e. every chief factor of $G$ below $P$ is cyclic if $\Phi (P) \neq 1$. Now assume that $\Phi (P)=1$. Then $P$ is elementary abelian and let $N$ be a minimal subgroup of $P$. Suppose that that $N$ has a supersoluble supplement $T$ in $G$. If $N \leq T$, then we have that $T = G$ is $p$-supersoluble and we are done. If $N \cap T =1$, it follows directly that $P = P \cap NT = N (P \cap T)$, and $P \cap T < P$. Since $T$ normalizes $T \cap P$, $N$ normalizes $T \cap P$ as $N \leq P$ and $P$ is abelian, we have that $G = NT$ normalizes $P  \cap T$ and so every chief factor of $G$ below $P \cap T$ is cyclic. Since $P /P \cap T$ is of prime order, it implies from generalised Jordan-Holder Theorem that every chief factor of $G$ below $P$ is cyclic. Now suppose that every minimal subgroup $N$ of $P$ has no supersoluble supplement in $G$, then every minimal subgroup $N$ is weakly $S\Phi$-supplemented in $G$. Now assume that every minimal subgroup $\langle x \rangle$ of $P$ is S-permutable in $G$. Let $Q$ be an arbitrary Sylow $q$-subgroup of $G$. Then we get that $\langle x \rangle Q = Q \langle x \rangle$. By minimality of $p$ and {{\cite[Chapter IV, Theorem 2.8]{HU1}}}, we have that $Q \langle x \rangle$ is $p$-nilpotent and so $[Q, \langle x \rangle]=1$. By the choice of $Q$, it indicates that $[G,\langle x \rangle]=1$. By the randomness of $\langle x \rangle$ and ${\rm{exp}} (P)=p$, we conclude that $[G,P]=1$, and so $P \leq Z(G)$. Therefore every subgroup of $P$ is  normal in $G$, hence every chief factor of $G$ below $P$ is cyclic.  Now we may assume that there exists a minimal subgroup $N_0 = \langle x_0 \rangle$ of $P$ such that $N_0$ is not S-permutable in $G$. Then $(N_0)_{sG} = 1$. Therefore we get that there exists $T \leq G$ such that $G = \langle x_0 \rangle T$, and $\langle x_0 \rangle \cap T \leq \Phi (\langle x_0 \rangle ) (N_0)_{sG} = 1$. As above, it follows that $P \cap T$ is a normal subgroup of $G$ such that $P \cap T < P$. Again, by induction we get that every chief factor of $G$ below $T \cap P$ is cyclic. Since $P = N(P\cap T)$, every chief factor of $G$ below $P$ is cyclic and the proof is complete.
\end{proof}
\section{Characterizations for  supersolvability of $\mathcal{F}_S (G)$}
In this section, we investigate the structure of $\mathcal{F}_S (G)$ under the assumption that every minimal subgroup of $S$ is weakly $S\Phi$-supplemented in $G$, and prove the Theorem \ref{66666}. The proof of the theorem strongly relies on the proceeding results we have obtained in Section \ref{3}.
\begin{proof}[Proof of Theorem \ref{66666}]
Assume that the theorem is false, and let $G$ be a counterexample of  minimal order. Now denote $\mathcal{F}_S (G)$ by $\mathcal{F}$.
\begin{itemize}
\item[\textbf{Step 1.}] Let $H$ be a proper subgroup of $G$ such that $S \cap H \in {\rm Syl}_p (H)$ and $|S \cap H| \geq p^2$. Then $\mathcal{F}_{S \cap H} (H)$ is supersolvable.
\end{itemize}

By our hypothesis, every minimal subgroup $T$ of $S \cap H$ is weakly $S\Phi$-supplemented in $G$. Then every cyclic subgroup $T$ of $S \cap H$ with order $p$ or $4$ (If $p =2$)is weakly $S\Phi$-supplemented in  $H$ by lemma \ref{1} (1). Notice that ${\rm exp} (S \cap H)=p$ as ${\rm exp} (S) = p$ and $H \cap S \leq S$, hence  $H$ satisfies the hypothesis of the theorem and it follows from the minimal choice of $G$ that $\mathcal{F}_{S \cap H} (H)$ is supersolvable.
\begin{itemize}
\item[\textbf{Step 2.}] Let $Q \in \mathcal{E}_{\mathcal{F}} ^{*}$, then $|Q| \geq p^2$. If moreover that $Q \not\unlhd G$, then $N_{\mathcal{F}} (Q)$ is supersolvable.
\end{itemize}

Suppose that there exists a subgroup $Q \in \mathcal{E}_{\mathcal{F}} ^{*}$ such that $|Q| < p^2$. Then there is a subgroup $R$ of $S$ such that $|R| = p $, and $Q <R$. It follows directly that $R \leq C_S (Q)$. Since $Q < R \leq S$, we conclude from $Q$ is a member of $\mathcal{E}_{\mathcal{F}} ^{*}$ that $Q$ is $\mathcal{F}$-essential. By the definition, $Q$ is $\mathcal{F}$-centric. Hence $R \leq C_S (Q) =Z(Q) \leq Q$, a contradiction. Thus $|Q| \geq p^2$.

Assume that $Q$ is not normal in $G$. Therefore $N_G (Q)$ is a proper subgroup of $G$. Since $Q \in \mathcal{E}_{\mathcal{F}} ^{*}$, $Q$ is fully $\mathcal{F}$-normalized or $Q=S$. Clearly $S$ is fully $\mathcal{F}$-normalized, hence $Q$ is always fully $\mathcal{F}$-normalized. By the argument below {{\cite[Definition 2.4]{AK}}}, $S \cap N_G (Q) = N_S (Q) \in {\rm Syl}_p (N_G(Q))$. Since $|N_S (Q)| \geq |Q| \geq p^2$, it yields that $N_G (Q)$ satisfies the hypothesis of Step 1, and so $\mathcal{F}_{N_S (Q)} (N_G(Q))=N_{\mathcal{F}} (Q)$ is supersolvable.
\begin{itemize}
\item[\textbf{Step 3.}] $|O_p (G)| \geq p^2$.
\end{itemize}

Assume that there does not exist a subgroup $Q \in \mathcal{E}_{\mathcal{F}} ^{*}$ such that $Q \unlhd G$. Then for each $Q \in  \mathcal{E}_{\mathcal{F}} ^{*}$, the fusion system $N_{\mathcal{F}} (Q)$ is supersolvable by Step 2. By Lemma \ref{111111}, $\mathcal{F}$ is supersolvable, a contradiction. Thus there exists a subgroup $Q \in \mathcal{E}_{\mathcal{F}} ^{*}$ such that $Q \unlhd G$. Hence we conclude from Step 2 that $|O_p (G)| \geq |Q| \geq p^2$.
\begin{itemize}
\item[\textbf{Step 4.}] $O_p (G) \leq Z_{\mathfrak{U}} (G)$.
\end{itemize}

It follows from $|O_p (G)| \geq p^2$ that  any minimal subgroup $T$ of $O_p (G)$   is weakly $S\Phi$-supplemented in $G$. Notice that ${\rm exp} (O_p (G)) = p$ as $O_p (G) \leq S$ and ${\rm exp} (S) = p$,  it yields from Theorem \ref{3.4} that every chief factor of $G$ below $P$ is cyclic. Therefore, for any chief factor $H / K$ below $O_p (G)$, it follows from the fact $H / K$ is cyclic  that $H / K$ is of order $p$. Consider the semidirect product $ U = H / K \rtimes G / C_G (H / K)$, then we conclude from $|G|$ is a multiple of   $|G /C_G (H / K)|$ that the subgroup $ G /C_G (H / K)$ has index $p$ in $U$, which is the smallest prime divisor of $|U|$ as well. Thus $G /C_G (H / K)$ is normal in $U$, and we obtain from the fact $G /C_G (H / K) \lesssim  {\rm Aut} (H / K)$ that $G /C_G (H / K)$ is abelian, and so $U$ is supersolvable. Hence by the choice of $H/K$, $O_p (G)$ is $\mathfrak{U}$-hypercentral in $G$, which indicates that   $O_p (G) \leq Z_{\mathfrak{U}} (G)$ and this part is complete.
\begin{itemize}
\item[\textbf{Step 5.}] Final contradiction.
\end{itemize}

Suppose that $H$ is a proper subgroup of $G$ such that $O_p (G) < S \cap H$ and $S \cap H \in {\rm Syl}_p (H)$. By Step 1 and Step 3, $|S \cap H| >|O_p (G)| \geq p^2$ and so $\mathcal{F}_{S \cap H} (H)$ is supersolvable. Since $O_p (G) \leq Z_{\mathfrak{U}} (G)$ by Step 4, it follows directly from Lemma \ref{222222} that $\mathcal{F}_S (G)$ is supersolvable, a contradiction. Hence our proof is complete.
\end{proof}
As a direct application of the theorem above, we obtain the following characterization for the structure of finite groups under the assumption that all minimal subgroups of a Sylow $p$-subgroup $S$ are weakly $S\Phi$-supplemented in $G$. 
\begin{corollary}
Let $G$ be a finite group and $S$ a Sylow $p$-subgroup of $G$, where $p$ is the smallest prime divisor of $|G|$. Suppose that ${\rm exp} (S) = p$, and every minimal subgroup of $S$  is weakly $S\Phi$-supplemented in $G$, then $G$ is $p$-nilpotent.
\end{corollary}
\begin{proof}
It follows from Theorem \ref{66666} that $\mathcal{F}_S (G)$  is supersolvable. Since $p$ is the smallest prime divisor of $|G|$, we conclude from {{\cite[Theorem 1.9]{SN}}} that $G$ is $p$-nilpotent, as desired.
\end{proof}
\!\!\!\!\!\!\!\!\!\textbf{Availability of data and materials:} Not applicable.
\\ \quad \\
\!\!\!\!\!\!\!\!\!\textbf{Ethical Approval:} Not applicable.
\\ \quad \\
\!\!\!\!\!\!\!\!\!\textbf{Declarations of interest:} Not applicable.
\\ \quad \\
\!\!\!\!\!\!\!\!\!\textbf{Funding:} Not applicable.


\begin{thebibliography}{0}

\bibitem{MA2} M. Asaad, The Influence of Weakly $s$-Supplemented Subgroups on the Structure of Finite Groups. {\it Comm. Algebra}, \textbf{42} (2014) 2319-2330.
\bibitem{MA} M. Asaad, The influence of weakly $S\Phi$-supplemented subgroups on the structure of finite groups. {\it Publ. Math. Debrecen}
\textbf{102} (2023) 189–195.
\bibitem{AK} M. Aschbacher, R. Kessar, B. Oliver, Fusion Systems in Algebra and Topology, Cambridge University Press, New York, 2011.

\bibitem{FJ} F. Aseeri, J. Kaspczyk, Criteria for supersolvability of saturated fusion systems. https://doi.org/10.48550/arXiv.2305.09008.

\bibitem{BA} A. Ballester-Bolinches, L. M. Ezquerro, A. N. Skiba,  Subgroups of finite groups with a strong cover-avoidance property.
{\it Bull. Aust. Math. Soc.} \textbf{79} (2009) 499-506 .
\bibitem{GL} X. Guo, J. Lu, On $SS$-supplemented subgroups of finite groups and their properties. {\it Glasgow Math. J.} \textbf{54} (2012) 481–491.

\bibitem{GW} W. Guo, The Theory of Classes of Groups. Kluwer Academic, Dordrecht, 2000.
\bibitem{HU1} B. Huppert, Endliche Gruppen I. Springer, New York, 1967.
\bibitem{HU3} B. Huppert, N. Blackburn, Finite Groups III. Springer, Berlin, 1982.
\bibitem{KS} H. Kurzweil, B. Stellmacher, The Theory of Finite Groups. Springer New York, 2004.
\bibitem{LX} J. Li, F. Xie,  On inequalities of subgroups and the structure of finite groups. {\it J. Inequal. Appl.} \textbf{427} (2013) 1-7.
\bibitem{LZ} X. Li, T. Zhao, $S\Phi$-supplemented subgroups of finite groups,  {\it Ukr. Math. J.} \textbf{64}  (2012) 102-109.
\bibitem{AS} A. N. Skiba, On weakly $s$-permutable subgroups of finite groups. {\it J. Algebra}, \textbf{315} (2007) 192–209.
\bibitem{SN} N. Su, On supersolvable saturated fusion systems, {\it Monatsh Math} \textbf{187} (2018) 171–179.
\bibitem{WY} Y. Wang, A remark on the $c$-normality of maximal subgroups of finite groups. {\it Proc. Edinb. Math. Soc.}  \textbf{40} (1997) 243-246.
\bibitem{WU} Z. Wu, Y. Mao, W. Guo, On weakly $S\Phi$-supplemented subgroups of finite groups, {\it Publ. Math. Debrecen}, \textbf{57} (2016) 696-703.



\end{thebibliography}
\end{document}